\documentclass[12pt]{amsart}       
\usepackage{txfonts}
\usepackage{amssymb}
\usepackage{eucal}
\usepackage{graphicx}
\usepackage{amsmath}
\usepackage{amscd}
\usepackage[all]{xy}           
\usepackage{tikz}
\usepackage{amsfonts,latexsym}
\usepackage{xspace}
\usepackage{epsfig}
\usepackage{float}
\usepackage{color}
\usepackage{fancybox}
\usepackage{colordvi}
\usepackage{multicol}
\usepackage{colordvi}
\usepackage[colorlinks,final,backref=page,hyperindex,hypertex]{hyperref}
\usepackage[active]{srcltx} 

\newtheorem{theorem}{Theorem}[section]
\newtheorem{prop}[theorem]{Proposition}
\theoremstyle{definition}
\newtheorem{defn}[theorem]{Definition}

\newtheorem{coro}[theorem]{Corollary}
\newtheorem{prop-def}{Proposition-Definition}[section]
\newtheorem{coro-def}{Corollary-Definition}[section]

\newtheorem{remark}[theorem]{Remark}

\newtheorem{exam}[theorem]{Example}

\newcommand{\nc}{\newcommand}
\nc{\tred}[1]{\textcolor{red}{#1}}
\nc{\tblue}[1]{\textcolor{blue}{#1}}
\nc{\tgreen}[1]{\textcolor{green}{#1}}
\nc{\tpurple}[1]{\textcolor{purple}{#1}}
\nc{\btred}[1]{\textcolor{red}{\bf #1}}
\nc{\btblue}[1]{\textcolor{blue}{\bf #1}}
\nc{\btgreen}[1]{\textcolor{green}{\bf #1}}
\nc{\btpurple}[1]{\textcolor{purple}{\bf #1}}
\nc{\NN}{{\mathbb N}}
\nc{\ncsha}{{\mbox{\cyr X}^{\mathrm NC}}} \nc{\ncshao}{{\mbox{\cyr
X}^{\mathrm NC}_0}}

\renewcommand{\frak}{\mathfrak}

\newcommand{\efootnote}[1]{}

\renewcommand{\textbf}[1]{}

\newcommand{\delete}[1]{}

\nc{\mlabel}[1]{\label{#1}}  
\nc{\mcite}[1]{\cite{#1}}  
\nc{\mref}[1]{\ref{#1}}  
\nc{\mbibitem}[1]{\bibitem{#1}} 

\delete{
\nc{\mlabel}[1]{\label{#1}  
{\hfill \hspace{1cm}{\small\tt{{\ }\hfill(#1)}}}}
\nc{\mcite}[1]{\cite{#1}{\small{\tt{{\ }(#1)}}}}  
\nc{\mref}[1]{\ref{#1}{{\tt{{\ }(#1)}}}}  
\nc{\mbibitem}[1]{\bibitem[\bf #1]{#1}} 
}

\nc{\rbp}{Rota-Baxter pair\xspace}
\nc{\rbps}{Rota-Baxter pairs\xspace}
\nc{\rbped}{Rota-Baxter paired\xspace}
\nc{\srbped}{special Rota-Baxter paired\xspace}
\nc{\grbped}{generic Rota-Baxter paired\xspace}

\nc{\tforall}{\text{ for all }}

\nc{\bin}[2]{ (_{\stackrel{\scs{#1}}{\scs{#2}}})}  
\nc{\binc}[2]{ \left (\!\! \begin{array}{c} \scs{#1}\\
    \scs{#2} \end{array}\!\! \right )}  
\nc{\bincc}[2]{  \left ( {\scs{#1} \atop
    \vspace{-1cm}\scs{#2}} \right )}  
\nc{\bs}{\bar{S}} \nc{\cosum}{\sqsubset} \nc{\la}{\longrightarrow}
\nc{\rar}{\rightarrow} \nc{\dar}{\downarrow} \nc{\dprod}{**}
\nc{\dap}[1]{\downarrow \rlap{$\scriptstyle{#1}$}}
\nc{\md}{\mathrm{dth}} \nc{\uap}[1]{\uparrow
\rlap{$\scriptstyle{#1}$}} \nc{\defeq}{\stackrel{\rm def}{=}}
\nc{\disp}[1]{\displaystyle{#1}} \nc{\dotcup}{\
\displaystyle{\bigcup^\bullet}\ } \nc{\gzeta}{\bar{\zeta}}
\nc{\hcm}{\ \hat{,}\ } \nc{\hts}{\hat{\otimes}}
\nc{\free}[1]{\bar{#1}}
\nc{\uni}[1]{\tilde{#1}} \nc{\hcirc}{\hat{\circ}} \nc{\lleft}{[}
\nc{\lright}{]} \nc{\lc}{\lfloor} \nc{\rc}{\rfloor}
\nc{\curlyl}{\left \{ \begin{array}{c} {} \\ {} \end{array}
    \right .  \!\!\!\!\!\!\!}
\nc{\curlyr}{ \!\!\!\!\!\!\!
    \left . \begin{array}{c} {} \\ {} \end{array}
    \right \} }
\nc{\longmid}{\left | \begin{array}{c} {} \\ {} \end{array}
    \right . \!\!\!\!\!\!\!}
\nc{\onetree}{\bullet} \nc{\ora}[1]{\stackrel{#1}{\rar}}
\nc{\ola}[1]{\stackrel{#1}{\la}}
\nc{\ot}{\otimes} \nc{\mot}{{{\boxtimes\,}}}
\nc{\otm}{\overline{\boxtimes}} \nc{\sprod}{\bullet}
\nc{\scs}[1]{\scriptstyle{#1}} \nc{\mrm}[1]{{\rm #1}}
\nc{\margin}[1]{\marginpar{\rm #1}}   
\nc{\dirlim}{\displaystyle{\lim_{\longrightarrow}}\,}
\nc{\invlim}{\displaystyle{\lim_{\longleftarrow}}\,}
\nc{\mvp}{\vspace{0.3cm}} \nc{\tk}{^{(k)}} \nc{\tp}{^\prime}
\nc{\ttp}{^{\prime\prime}} \nc{\svp}{\vspace{2cm}}
\nc{\vp}{\vspace{8cm}} \nc{\proofbegin}{\noindent{\bf Proof: }}
\nc{\proofend}{$\blacksquare$ \vspace{0.3cm}}
\nc{\modg}[1]{\!<\!\!{#1}\!\!>}
\nc{\intg}[1]{F_C(#1)} \nc{\lmodg}{\!
<\!\!} \nc{\rmodg}{\!\!>\!}
\nc{\cpi}{\widehat{\Pi}}
\nc{\sha}{{\mbox{\cyr X}}}  
\nc{\ssha}{{\mbox{\cyrs X}}} 
\nc{\shpr}{\diamond}    
\nc{\shp}{\ast} \nc{\shplus}{\shpr^+}
\nc{\shprc}{\shpr_c}    
\nc{\msh}{\ast} \nc{\zprod}{m_0} \nc{\oprod}{m_1}
\nc{\labs}{\mid\!} \nc{\rabs}{\!\mid}
\nc{\sqmon}[1]{\langle #1\rangle}

\nc{\mmbox}[1]{\mbox{\ #1\ }} \nc{\dep}{\mrm{dep}} \nc{\fp}{\mrm{FP}}
\nc{\rchar}{\mrm{char}} \nc{\End}{\mrm{End}} \nc{\Fil}{\mrm{Fil}}
\nc{\Mor}{Mor\xspace} \nc{\gmzvs}{gMZV\xspace}
\nc{\gmzv}{gMZV\xspace} \nc{\mzv}{MZV\xspace}
\nc{\mzvs}{MZVs\xspace} \nc{\Hom}{\mrm{Hom}} \nc{\id}{\mrm{id}}
\nc{\im}{\mrm{im}} \nc{\incl}{\mrm{incl}} \nc{\map}{\mrm{Map}}
\nc{\mchar}{\rm char} \nc{\nz}{\rm NZ} \nc{\supp}{\mathrm Supp}

\nc{\Alg}{\mathbf{Alg}} \nc{\Bax}{\mathbf{Bax}} \nc{\bff}{\mathbf f}
\nc{\bfk}{{\bf k}} \nc{\bfone}{{\bf 1}} \nc{\bfx}{\mathbf x}
\nc{\bfy}{\mathbf y}
\nc{\base}[1]{\bfone^{\otimes ({#1}+1)}} 
\nc{\Cat}{\mathbf{Cat}}

\nc{\detail}{\marginpar{\bf More detail}
    \noindent{\bf Need more detail!}
    \svp}
\nc{\Int}{\mathbf{Int}} \nc{\Mon}{\mathbf{Mon}}
\nc{\rbtm}{{shuffle }} \nc{\rbto}{{Rota-Baxter }}
\nc{\remarks}{\noindent{\bf Remarks: }} \nc{\Rings}{\mathbf{Rings}}
\nc{\Sets}{\mathbf{Sets}} \nc{\wtot}{\widetilde{\odot}}
\nc{\wast}{\widetilde{\ast}} \nc{\bodot}{\bar{\odot}}
\nc{\bast}{\bar{\ast}} \nc{\hodot}[1]{\odot^{#1}}
\nc{\hast}[1]{\ast^{#1}} \nc{\mal}{\mathcal{O}}
\nc{\tet}{\tilde{\ast}} \nc{\teot}{\tilde{\odot}}
\nc{\oex}{\overline{x}} \nc{\oey}{\overline{y}}
\nc{\oez}{\overline{z}} \nc{\oef}{\overline{f}}
\nc{\oea}{\overline{a}} \nc{\oeb}{\overline{b}}
\nc{\weast}[1]{\widetilde{\ast}^{#1}}
\nc{\weodot}[1]{\widetilde{\odot}^{#1}} \nc{\hstar}[1]{\star^{#1}}
\nc{\lae}{\langle} \nc{\rae}{\rangle}
\nc{\lf}{\lfloor}
\nc{\rf}{\rfloor}

\nc{\QQ}{{\mathbb Q}}
\nc{\RR}{{\mathbb R}} \nc{\ZZ}{{\mathbb Z}}

\nc{\cala}{{\mathcal A}} \nc{\calb}{{\mathcal B}}
\nc{\calc}{{\mathcal C}}
\nc{\cald}{{\mathcal D}} \nc{\cale}{{\mathcal E}}
\nc{\calf}{{\mathcal F}} \nc{\calg}{{\mathcal G}}
\nc{\calh}{{\mathcal H}} \nc{\cali}{{\mathcal I}}
\nc{\call}{{\mathcal L}} \nc{\calm}{{\mathcal M}}
\nc{\caln}{{\mathcal N}} \nc{\calo}{{\mathcal O}}
\nc{\calp}{{\mathcal P}} \nc{\calr}{{\mathcal R}}
\nc{\cals}{{\mathcal S}} \nc{\calt}{{\mathcal T}}
\nc{\calu}{{\mathcal U}} \nc{\calw}{{\mathcal W}} \nc{\calk}{{\mathcal K}}
\nc{\calx}{{\mathcal X}} \nc{\CA}{\mathcal{A}}

\nc{\fraka}{{\mathfrak a}} \nc{\frakA}{{\mathfrak A}}
\nc{\frakb}{{\mathfrak b}} \nc{\frakB}{{\mathfrak B}}
\nc{\frakD}{{\mathfrak D}} \nc{\frakF}{\mathfrak{F}}
\nc{\frakf}{{\mathfrak f}} \nc{\frakg}{{\mathfrak g}}
\nc{\frakH}{{\mathfrak H}} \nc{\frakL}{{\mathfrak L}}
\nc{\frakM}{{\mathfrak M}} \nc{\bfrakM}{\overline{\frakM}}
\nc{\frakm}{{\mathfrak m}} \nc{\frakP}{{\mathfrak P}}
\nc{\frakN}{{\mathfrak N}} \nc{\frakp}{{\mathfrak p}}
\nc{\frakS}{{\mathfrak S}} \nc{\frakT}{\mathfrak{T}}
\nc{\frakX}{\mathfrak{X}} \nc{\frakx}{\mathfrak{x}}
\nc{\fraky}{\mathfrak{y}}

\nc{\BS}{\mathbb{S}} \nc{\bfT}{\mathbf{T}}

\font\cyr=wncyr10 \font\cyrs=wncyr7
\nc{\li}[1]{\textcolor{red}{Li:#1}}
\nc{\ly}[1]{\textcolor{blue}{Liangyun: #1}}
\nc{\hh}[1]{\textcolor{purple}{Huihui:#1}}
\nc{\UN}{U_{N}} \nc{\lbar}[1]{\overline{#1}}
\nc{\FN}{F_{\kappa}} \nc{\FNN}{F_{-\lambda^2}}
\nc{\bre}{{\rm bre}} \nc{\w}{{\rm bre}}
\nc{\altx}{\Lambda}
\nc{\spr}{\cdot} \nc{\hma}{\mathcal{H}}
\nc{\rts}{\stackrel{\rightarrow}{\shpr}}
\nc{\ox}{\overline{\frak x}}
\nc{\oX}{\overline{X}} \nc{\mt}{T}
\nc{\mtw}{{\rm MTW}} \nc{\vep}{\varepsilon}
\nc{\mop}{P_A}

\begin{document}

\title{Rota-Baxter paired modules and their constructions from Hopf algebras}
%
\author{Huihui Zheng}
\address{Department of Mathematics, Nanjing Agricultural University,
Nanjing 210095, China}
\email{huihuizhengmail@126.com}

\author{Li Guo}
\address{Department of Mathematics and Computer Science,
         Rutgers University,
         Newark, NJ 07102, USA}
\email{liguo@rutgers.edu}

\author{Liangyun Zhang}
\address{Department of Mathematics, Nanjing Agricultural University,
Nanjing 210095, China}
\email{zlyun@njau.edu.cn}

\date{\today}
\begin{abstract}
In this paper, we introduce the concept of a \rbped module to study Rota-Baxter modules without necessarily a Rota-Baxter operator.
We obtain two characterizations of Rota-Baxter paired modules, and give some basic properties of Rota-Baxter paired modules. Beginning with the connection between the notion of integrals in the representations of Hopf algebra and of the notion of an integral algebra as a motivation and special case of Rota-Baxter algebra of weight zero, we obtain a large number of \rbped modules from Hopf related algebras and modules, including  semisimple Hopf algebras, weak Hopf algebras, Long bialgebras, quasitriangular Hopf algebras, weak Hopf modules, dimodules and Doi-Hopf modules. Some of them give new examples of Rota-Baxter algebras.
\end{abstract}

\subjclass[2010]{16T05, 16W99}

\keywords{Rota-Baxter algebra, Rota-Baxter module, Hopf algebra, dimodule,
Hopf module}

\maketitle

\tableofcontents

\setcounter{section}{0}

\allowdisplaybreaks

\section{ Introduction}

A {\bf Rota-Baxter algebra} (first known as a Baxter algebra) is an
algebra $A$ with a linear operator $P$ on $A$ that satisfies the {\bf Rota-Baxter identity}
\begin{equation}
P(x)P(y)=P(P(x)y)+P(xP(y))+\lambda P(xy)\ \ \tforall x,y\in A,
\mlabel{eq:rbo}
\end{equation}
where $\lambda$, called the {\bf weight}, is a fixed element in the base ring of the algebra $A$~\cite{G. Baxter1960,G.C. Rota1969}.

In the 1960s, Rota began a systematic study of Rota-Baxter algebras from an
algebraic and combinatorial perspective in connection with
hypergeometric functions, incidence algebras and symmetric functions~\cite{G.C. Rota1969}. In the Lie algebra context, the Rota-Baxter operator of weight zero was rediscovered by the physicists in the 1980s as the operator form of the classical Yang-Baxter equation~\mcite{STS}.

At the turn of this century, Rota-Baxter algebra, together with Hopf algebra and pre-Lie algebra, appeared as the main algebraic structures underlying the Connes-Kreimer approach~\mcite{CK1} of renormalization in quantum field theory.
Hopf algebra is also one of the main directions in the theoretical study of Rota-Baxter algebra. In 2003, Andrews et al.~\cite{F.V. Andrew2003} constructed a new class of Hopf algebras on free Rota-Baxter algebras by using combinatorial methods. Later study of Rota-Baxter algebras with examples and connections to Hopf algebras can be found in \cite{K. Ebrahimi2006}, \cite{R.Q. Jian2014} and
\cite{T. Ma2016}. Recently, Brzezi$\acute{n}$ski ~\cite{T. Brzezinski2016} and Gao-Guo-Zhang ~\cite{X. Gao2016} established the relationships between Rota-Baxter algebras and Hopf algebras by covariant functors and cocyles respectively.

To study the representations of Rota-Baxter algebras, the concept a Rota-Baxter module was introduced in~\mcite{GL} and was related to the ring of Rota-Baxter operators. By definition, a {\bf Rota-Baxter module} over a Rota-Baxter algebra $(A,P)$ is a pair $(M,T)$ where $M$ is a (left) $R$-module and $T:M\to M$ is a $\bfk$-linear operator such that
\begin{equation}
P(a)\cdot T(m)=T(P(a)\cdot m)+T(a\cdot T(m))+\lambda T(a \cdot m) \ \ \tforall a\in A, m\in M.
\mlabel{eq:rbm}
\end{equation}
Further studies in this direction were pursued in~\mcite{GGQ,LQ,QP} on regular-singular decompositions, geometric representations and derived functors of Rota-Baxter modules, especially those over the Rota-Baxter algebras of Laurent series and polynomials.

The present study stemmed from the observation that, in the defining equation~\eqref{eq:rbm} of a Rota-Baxter operator $T$ on a module $M$, the axiom in Eq.~\eqref{eq:rbo} of a Rota-Baxter operator $P$ on an algebra $R$ is not strictly needed. This suggests that Rota-Baxter modules can be considered in a much broader context in which the operator $T$ on the module $M$ is required to satisfy Eq.~\eqref{eq:rbm}, without requiring $P$ to satisfy Eq.~\eqref{eq:rbo}. This lead us to the concept of a {\bf \rbped module}, meaning that $P$ is paired with $T$ to form a Rota-Baxter module. This gives a natural generalization of Rota-Baxter modules.

Two aspects of \rbped suggest that the benefit of studying \rbped modules. On the one hand, many properties of Rota-Baxter modules, even of Rota-Baxter algebras, are naturally generalized to \rbped modules. Of course, there are properties which are completely new to \rbped modules. In fact, under suitable linearity and idempotent properties of a linear operator $T$ on a module, we show that the operator gives a \rbped module regardless the operator $P$ on the algebra (Theorem~\mref{thm:3.2}), leading to the concept of a {\bf \grbped module}.

On the other hand, the more relaxed condition on \rbped allows it to have broader connections and applications, especially to Hopf algebras. A case in point is the notion of an integral, a notion which appears both in Rota-Baxter algebras and to Hopf algebras with different meanings. The integral is one of the motivations and special cases of Rota-Baxter algebras (of weight zero). It is also a basic tool to study the representations of Hopf algebras. As we will see (Corollary~\mref{co:int}), integrals in Hopf algebras are naturally related to \rbped modules, of weight -1 instead of zero.

We present these two aspects in the two main sections of this paper.
In Section~\mref{sec:rbp}, we introduce the basic concept of \rbped modules and give some preliminary examples of \rbped modules, for example, on the endomorphism algebra $\End (M)$ of a \rbped module $M$.
We obtain two characterizations of \rbped modules. One is a generalization of the Atkinson factorization~\cite{F.V. Atkinson,L. Guo2012}. One is new for \rbped module under the assumption of quasi-idempotency. Further properties of \rbped modules are also provided. In Section~\mref{sec:const}, we construct a large number of \rbped modules from Hopf algebra related structures. Along the way, we obtain new examples of Rota-Baxter algebras. We first apply a linear functional, especially an integral, on a Hopf algebra $H$ to give a \rbped module structure on $H$. We also obtain \rbped modules on $H$ from a target map on a weak bialgebra $H$ and an antipode on a quantum commutative weak Hopf algebra $H$. We then consider modules over a Hopf algebra and equip the modules with \rbped module structures via the antipodes. On a module which is left module over a bialgebra $H$ and at the same time a right comodule over $H$ (a left, right $H$-dimodule), \grbped modules are obtained from linear functional on $H$ which are idempotent under the convolution product, with special attention given to cointegrals, skew pairing Long bialgebras, braided bialgebras and quasitriangular bialgebras. Finally we build \rbped module structures from weak right Doi-Hopf modules on a weak Hopf algebra.
\smallskip

\noindent

{\bf Notations.} Throughout the paper, all algebras, linear maps and tensor products are taken over a fixed commutative ring $\bfk$ unless otherwise specified. Likewise, an algebra means a unitary one and a module means a left unitary one.

\section{\rbped modules}
\mlabel{sec:rbp}
In this section, we introduce the concept of a \rbp module, its characterizations and other properties.

\subsection{\rbped modules and their characterizations}
For a Rota-Baxter algebra $(A,P)$ defined in Eq.~\eqref{eq:rbo},
the linear map $P$ is called a {\bf Rota-Baxter operator of weight $\lambda$}.

We refer the reader to~\cite{L. Guo2012} for further discussions on Rota-Baxter algebras and only give the following simple examples of Rota-Baxter algebras which we will be revisited later.
\begin{exam}\mlabel{ex:1.2}
\begin{enumerate}
\item
Let $A$ be an algebra. If $A$ is a augmented algebra~\cite{L. Positselski1995} in the sense that there exists an algebra homomorphism $f: A\rightarrow \bfk$, then $(A,P_f)$ is a Rota-Baxter algebra of weight $-1$, where the operator $P_f$ is given by
$$
P_f: A\rightarrow A,\ \ a\mapsto f(a)1_A \ \ \tforall a\in A.
$$
\item
Fix an $r\in A$ and define a map
$$P: A\rightarrow A, \quad P(a)=ra \ \ \text{ for all } a\in A. $$
Then, $(A,P)$ is a Rota-Baxter algebra with weight $-1$ if only if $r^2=r$ since Eq.~(\ref{eq:rbo}) is equivalent to $r^2ab-rab=0$ for all $a,b\in A$, in particular for $a=b=1$.
Take $A$ to be the $n\times n$ matrix algebra M$_{n}(\RR)$ with entries in the real number field and, for any $j=1,2,\cdots,n$, let $E_{jj}$ denotes the $n\times n$ matrix whose $(j,j)$-entry is $1$ and other entries are $0$. Then we have Rota-Baxter algebras (M$_{n}(\RR), P_j)$ of weight $-1$, where
$$P_j:\mathrm{M}_{n}(\RR)\rightarrow \mathrm{M}_{n}(\RR), \quad
M\mapsto E_{jj}M \text{ for all } M\in \mathrm{M}_n(\RR).
$$
\end{enumerate}
\end{exam}

Now we introduce the key concept of this paper.

\begin{defn}
Fix a $\lambda\in \bfk$.
Let $A$ be an algebra and $M$ a left $A$-module. A pair $(P,T)$ of linear maps $P: A\rightarrow A$ and $T: M\rightarrow M$ is called a {\bf \rbped operator} of weight $\lambda$ on $(A,M)$ or simply on $M$ if
\begin{equation}
P(a)\cdot T(m)=T(P(a)\cdot m)+T(a\cdot T(m))+\lambda T(a\cdot m) \quad \text{ for all } a\in A, m\in M.
\mlabel{eq:rbp}
\end{equation}
We then call the triple $(M,P,T)$ a {\bf \rbped $A$-module} of weight $\lambda$. We also called the pair $(M,T)$ a {\bf \srbped $A$-module} of weight $\lambda$, matched with $P$.

For a given linear operator $T:M\to M$, if for every linear map $P: A\rightarrow A$, $(M,P,T)$ is a \rbped $A$-module of weight $\lambda$, then $(M,T)$ is called a {\bf \grbped $A$-module} of weight $\lambda$.
\mlabel{de:2.1}
\end{defn}

A {\bf \rbped $A$-submodule} $N$ of a \rbped $A$-module $(M,P,T)$ is an $A$-submodule $N$ of $M$ such that $T(N)\subseteq N$. Then $N$ is a \rbped $A$-module. A {\bf Rota-Baxter module map} $f:(M,P,T)\rightarrow (M',P',T')$ of the same weight $\lambda$ is a $A$-module map such that $f\circ T=T'\circ f$. Then $\ker f$ and $\im f$ are \rbped $A$-submodules of $M$ and $M'$ respectively.
The same notions can be defined for {\bf right (generic) \rbped $A$-module}, etc.

\begin{exam}\mlabel{ex:2.2}
 \begin{enumerate}
 \item
Let $A$ be an algebra, regarded also as a left $A$-module. If $(A, P)$ is a Rota-Baxter algebra of weight $\lambda$, then, $(A,P,P)$ is a \rbped $A$-module of weight $\lambda$.
In particular, let $A$ be a augmented algebra. Then, by Example~\mref{ex:1.2}, $(A,P_f,P_f)$ is a \rbped $A$-module of weight $-1$.
\item By Example~\mref{ex:1.2} and the previous example, we see that (M$_{n}(\RR), P_j, P_j)$ is a Rota-Baxter paired M$_{n}(\RR)$-module, for any $j=1,2,\cdots,n$, where $P_j:$ M$_{n}(\RR)\rightarrow$ M$_{n}(\RR)$ is given by
$$
P_j(M)=E_{jj}M.
$$
\end{enumerate}
\end{exam}

In what follows, we give conditions for an $A$-module to be a \rbped module or a \grbped module.

Recall that a linear operator $T:M\to M$ is called {\bf quasi-idempotent of weight $\lambda$} if $T^2=-\lambda T$~\cite{AM,GL}. We have the following characterization of \grbped modules.

\begin{theorem}
Let $M$ be a left $A$-module.
\begin{enumerate}
\item
Let $T:M\to M$ be $A$-linear. Then the following statements are equivalent.
\begin{enumerate}
\item
$(M,T)$ is a \grbped $A$-module of weight $\lambda$;
\mlabel{it:3.2-1.1}
\item
There is a linear operator $P:A\to A$ such that $(M,P,T)$ is a \rbped $A$-module of weight $\lambda$. That is, $(M,T)$ is a \srbped $A$-module for some $P:A\to A$.
\mlabel{it:3.2-1.2}
\item
$T$ is quasi-idempotent of weight $\lambda$;
\mlabel{it:3.2-1.3}
\end{enumerate}
\mlabel{it:3.2-1}
\item
Let $T:M\to M$ be $\bfk$-linear. Denote
$$C_M:=\{a\in A\,|\, T(a\cdot m)=a\cdot T(m)\}.$$ Then $C_M$ is a subalgebra of $A$ and $T$ is $C_M$-linear.
\mlabel{it:3.2-2}
\item
Let $T:M\to M$ be $\bfk$-linear. The following statements are equivalent. \begin{enumerate}
\item
$(M,T)$ is a \grbped $C_M$-module of weight $\lambda$;
\item
There is a linear operator $P:C_M\to C_M$ such that $(M,P,T)$ is a \rbped $C_M$-module of weight $\lambda$;
\item
$T$ is quasi-idempotent of weight $\lambda$.
\end{enumerate}
\mlabel{it:3.2-3}
\end{enumerate}
\mlabel{thm:3.2}
\end{theorem}

\begin{proof}
(\mref{it:3.2-1}) Under the $A$-linearity condition of $T$, for any given linear operator $P:A\to A$, Eq.~(\mref{eq:rbp}) is equivalent to $xT^2(m)+\lambda xT(m)=0$ for all $x\in A$ and $m\in M$, which is equivalent to $T^2(m)+\lambda T(m)=0$ for all $m\in M$ since $A$ is unitary. Thus we have ((\mref{it:3.2-1.1}) $\Longleftrightarrow$ (\mref{it:3.2-1.3})). In a similar way, we can prove that ((\mref{it:3.2-1.2}) $\Longleftrightarrow$ (\mref{it:3.2-1.3})).

\smallskip

\noindent
(\mref{it:3.2-2}) It is straightforward to check that $C_M$ is a subalgebra of $A$. Then the $C_M$-linearity of $T$ follows from Item~(\mref{it:3.2-1}) in replacing $A$ by $C_M$.
\smallskip

\noindent
(\mref{it:3.2-3}) follows from combining Items~ (\mref{it:3.2-1}) and (\mref{it:3.2-2}).
\end{proof}

See Section~\mref{sec:const} for application of this theorem to the constructions of \rbped modules from Hopf algebras.

We next generalize to \rbped modules the Atkinson factorization~\mcite{F.V. Atkinson} as presented in~\mcite{GL}.

\begin{theorem}\mlabel{thm:3.5} Let $M$ be an $A$-module and let $P: A\rightarrow A$, $T:M\rightarrow M$ be linear maps. Let $\lambda\in \bfk$ have no zero divisors in $M$ and define $\widetilde{T},\widetilde{P}$ as in Eq.(5). Then, $(M, P, T)$ is a \rbped $A$-module of weight $\lambda\neq 0$ if and only if, for any given $a\in A$ and $m\in M$, there is $n\in M$ such that
\begin{equation}
P(a)\cdot T(m)=T(n), \quad \widetilde{P}(a)\cdot \widetilde{T}(m)=-\widetilde{T}(n).
\mlabel{eq:atk}
\end{equation}
\end{theorem}

\begin{proof}
Let $(M, P, T)$ be a \rbped $A$-module of weight $\lambda$. Then, for any $a\in A, m\in M$, we have
$$
P(a)\cdot T(m)=T(P(a)\cdot m)+T(a\cdot T(m))+\lambda T(a\cdot m).
$$
Let $n=P(a)\cdot m+a\cdot T(m)+\lambda (a\cdot m)$. Then the above equality gives
$$
P(a)\cdot T(m)=T(n). $$
From the proof of Proposition~\mref{pp:3.1}, we obtain
$$\widetilde{P}(a)\cdot \widetilde{T}(m)=-\widetilde{T}(n).$$

Now we consider the converse. For any $a\in A, m\in M$. Suppose that there exists $n\in M$ such that
$P(a)\cdot T(m)=T(n)$ and $\widetilde{P}(a)\cdot \widetilde{T}(m)=-\widetilde{T}(n)$. Then we have
 $$\begin{array}{rllr}
-\lambda n&=T(n)+\widetilde{T}(n)=P(a)\cdot T(m)-\widetilde{P}(a)\cdot \widetilde{T}(m)\\
 &=P(a)\cdot T(m)-((\lambda a+P(a))\cdot (\lambda m+T(m)))\\
 &=P(a)\cdot T(m)-(\lambda^2(a\cdot m)+\lambda(a\cdot T(m))+\lambda(P(a)\cdot m)+P(a)\cdot T(m))\\
 &=-\lambda(P(a)\cdot m+a\cdot T(m)+\lambda (a\cdot m)).
\end{array}$$
So we get
 $$\begin{array}{rllr}
n&=P(a)\cdot m+a\cdot T(m)+\lambda (a\cdot m),\\
P(a)\cdot T(m)&=T(n)=T(P(a)\cdot m)+T(a\cdot T(m))+\lambda T(a\cdot m),
\end{array}$$
 as needed.
 \end{proof}

\subsection{Other properties of Rota-Baxter paired modules}
\mlabel{ss:basic}
We now generalize basic properties of Rota-Baxter algebras and Rota-Baxter modules to \rbped modules.

Let $(M,P,T)$ be an \rbped $A$-module of weight $\lambda$. First it follows directly from the \rbped module axiom~\eqref{eq:rbp} that $T(M)$ is closed under the action of $P(A)$. Next for $\mu\in \bfk$, $(M, \mu P,\mu T)$ is a \rbped $A$-module of weight $\lambda\mu$. We also have the following generalization of Rota-Baxter algebras and Rota-Baxter modules.

\begin{prop}\mlabel{pp:2.5} Let $(M_i,P,T_i), i\in I,$ be a family of \rbped $A$-modules of weight $\lambda$. Define
$$ \mathbf{T}: \bigoplus_{i\in I} M_i \to \bigoplus_{i\in I}M_i, \quad
(m_i)\mapsto (T_i(m_i)) \ \ \tforall (m_i)\in \bigoplus_{i\in I}M_i.$$
Then $(\oplus_{i\in I} M_i, P, \mathbf{T})$ is a \rbped $A$-module of weight $\lambda$.
\end{prop}

\begin{proof} For any $a\in A, (m_i)\in \oplus_{i\in I} M_i$, we have
 \begin{eqnarray*}
 && P(a)\cdot \mathbf{T}(m_i)=P(a)\cdot (T_i(m_i))=(P(a)\cdot T_i(m_i))\\
 &&=(T_i(P(a)\cdot m_i))+T_i(a\cdot T_i(m_i))+\lambda T_i(a\cdot m_i)\\
 &&=\mathbf{T}(((P(a)\cdot m_i)))+\mathbf{T}((a\cdot T_i(m_i)))+\lambda \mathbf{T}((a\cdot m_i))\\
 &&=\mathbf{T}(P(a)\cdot (m_i))+\mathbf{T}(a\cdot \mathbf{T}((m_i))+\lambda \mathbf{T}(a\cdot (m_i)).
 \end{eqnarray*}
Thus $(\oplus_{i\in I}M_i, P, \mathbf{T})$ is a Rota-Baxter $A$-module of weight $\lambda$.
\end{proof}

It follows that the category of \rbped $A$-modules of weight $\lambda$ matched with a fixed linear operator $P:A\to A$ is an abelian category.

\begin{prop}\mlabel{pp:3.1}
Let $(M,P,T)$ be a \rbped $A$-module of weight $\lambda$. Define
\begin{equation}
\widetilde{P}=-\lambda id-P,\ \ \widetilde{T}=-\lambda id-T.
\mlabel{eq:tilde}
\end{equation}
Then, $(M, \widetilde{P},\widetilde{T})$ is also a \rbped $A$-module of weight $\lambda$.
\end{prop}

\begin{proof}
Let $(M,P,T)$ be a \rbped $A$-module of weight $\lambda$. Then, for any $a\in A, m\in M$,
$$\begin{array}{rllr}
 \widetilde{P}(a)\cdot\widetilde{T}(m)&=(-\lambda id-P)(a)\cdot (-\lambda id-T)(m)\\
 &=(-\lambda a-P(a))\cdot (-\lambda m-T(m))\\
 &= \lambda a \cdot\lambda m+P(a) \cdot\lambda m+\lambda a \cdot T(m)+P(a) \cdot T(m)\\
 &=\lambda(\lambda (a \cdot m)+P(a) \cdot m+ a \cdot T(m))+T(P(a)\cdot m)\\
 &\ \ \ +T(a\cdot T(m))+ \lambda T(a\cdot m)\\
 &=(\lambda id+T)((a\cdot T(m)+P(a)\cdot m+\lambda (a\cdot m))\\
 &=\widetilde{T}(-a\cdot T(m)-P(a)\cdot m-\lambda (a\cdot m))\\
 &=\widetilde{T}((-\lambda a-P(a))\cdot m+a\cdot (-\lambda m-T(m))+\lambda (a\cdot m))\\
 &=\widetilde{T}(\widetilde{P}(a)\cdot m+a\cdot \widetilde{T}(m)+\lambda (a\cdot m))\\
 &=\widetilde{T}(\widetilde{P}(a)\cdot m)+\widetilde{T}(a\cdot \widetilde{T}(m))+\lambda \widetilde{T}(a\cdot m)).
\end{array}$$
This is what we need.
\end{proof}

The following result shows that a \rbped structure can be extended to the endomorphism ring.

\begin{prop}\mlabel{pp:2.3}
Let $(M,P,T)$ be a right \rbped $A$-module of weight $\lambda$. Consider the left $A$-module $\End(M)$ with the left $A$-action given by
$$(a\cdot f)(m)=f(m\cdot a) \ \  \tforall a\in A, m\in \End(M).$$
Define
$$\bfT: \End(M)\rightarrow \End(M), \quad \bfT(f)(m)=f(T(m)), \tforall f\in \End(M), m\in M, $$
and define $\widetilde{\bfT}$ as in Eq.~\eqref{eq:tilde}. Then $(\End(M),P,\widetilde{\bfT})$ is a left \rbped $A$-module of weight $\lambda$.
\end{prop}

\begin{proof}
For any $a\in A, f\in \End(M), m\in M,$ we have
\begin{equation}
\begin{array}{rllr}
(P(a)\cdot \widetilde{\bfT}(f))(m)&=\widetilde{\bfT}(f)(m\cdot P(a))=-\lambda f(m\cdot P(a))-f(T(m\cdot P(a))),\\
\widetilde{\bfT}(P(a)\cdot f)(m)&=-\lambda (P(a)\cdot f)(m)-(P(a)\cdot f)(T(m))\\
&=-\lambda f(m\cdot P(a)) -f(T(m)\cdot P(a)),\\
\widetilde{\bfT}(a\cdot \widetilde{\bfT}(f))(m)
&=-\lambda (a\cdot \widetilde{\bfT}(f))(m)-(a\cdot\widetilde{\bfT}(f))(T(m))\\ &=-\lambda \widetilde{\bfT}(f)(m\cdot a) - \widetilde{\bfT}(f)(T(m)\cdot a)\\
&=\lambda^2 f(m\cdot a)) + \lambda f(T(m\cdot a))+\lambda f(T(m)\cdot a)+
(f T)(T(m)\cdot a),\\
\lambda \widetilde{\bfT}(a\cdot f)(m)&= -\lambda^2 (a\cdot f)(m)-\lambda(a\cdot f)(T(m))\\
&= -\lambda^2 f(m\cdot a)- \lambda f(T(m)\cdot a ).
\end{array}
\mlabel{eq:tt}
\end{equation}
Sum over the last three equations and cancel terms on the right hand side. Applying Eq~\eqref{eq:rbp} for the right \rbped module $(M,P,T)$, we obtain
$$P(a)\cdot\widetilde{\bfT}(f)=\widetilde{\bfT}(P(a)\cdot f)+\widetilde{\bfT}(a\cdot\widetilde{\bfT}(f))+\lambda \widetilde{\bfT}(a\cdot f). $$
Thus ($\End (M),P,\widetilde{\bfT})$ is a \rbped $A$-module of weight $\lambda$.
\end{proof}

\begin{exam}\mlabel{ex:2.4}
\begin{enumerate}
 \item
 Suppose that $(A,P)$ is a \rbped algebra of weight $0$. Then by Example 2.2, we obtain a \rbped $A$-module $(A,P,P)$ of weight $0$. So by the above proposition, ($\End (A),P,T_P)$ is a \rbped $A$-module of weight $0$, where $T_P:$ $\End (A)\rightarrow$ $\End (A)$ is defined by
$$T_P(f)(a)=-f(P(a)),
$$
for $f\in$$\End (A)$ and $a\in A$.
\mlabel{it:2.4-1}
\item
 Let $A$ be the $\mathbb{R}$-algebra of continuous functions on $\mathbb{R}$. Define a map $P: A\rightarrow A$ by the integration
 $$
 P(f)(x)=\int_0^xf(t)dt.
 $$
Then $(A,P)$ is a Rota-Baxter algebra of weight $0$~\cite[Example~1.1.4]{L. Guo2012}. So by (\mref{it:2.4-1}), we see that ($\End (A),P,T_P)$ is a \rbped $A$-module of weight $0$, where $T_P:$ $\End (A)\rightarrow$ $\End (A)$ is defined by $$T_P(\alpha)(f)=-\alpha(P(f)) \ \ \tforall \alpha\in\End (A),f\in A,$$
 that is,
$$T_P(\alpha)(f)(x)=-\alpha\left(\int_0^xf(t)dt\right).$$
\end{enumerate}
\end{exam}

The following result generalizes a property of Rota-Baxter algebra proved in~\mcite{AM}.
\begin{prop}\mlabel{pp:3.3}
 Let $(M, P, T)$ be a \rbped $A$-module of weight $\lambda$.
\begin{enumerate}
\item If $T$ is idempotent, then
\begin{equation}
\mlabel{eq:3.3-1}(1+\lambda)T(a\cdot T(m))=0 \ \ \tforall a\in A, m\in M.
\end{equation}
\mlabel{it:3.3-1}
\item
If $P$ and $T$ are idempotent, then
\begin{equation}
\mlabel{eq:3.3-2}(1+\lambda)T(P(a)\cdot m)=0 \ \ \tforall a\in A, m\in M.
\end{equation}
\mlabel{it:3.3-2}
\item
If $P$ and $T$ are idempotent, then
\begin{equation}
\mlabel{eq:3.3-3} (1+\lambda)(P(a)\cdot T(m)-\lambda T(a\cdot m))=0 \ \ \tforall a\in A, m\in M.
\end{equation}
\mlabel{it:3.3-3}
\end{enumerate}
\end{prop}

\begin{proof}
(\mref{it:3.3-1})
Since $T^2=T$, from Eq.~\eqref{eq:rbp} we obtain
 $$\begin{array}{rllr}
P(a)\cdot T(m)&=P(a)\cdot T^2(m)=P(a)\cdot T(T(m))\\
&=T(P(a)\cdot T(m))+T(a\cdot T(T(m)))+\lambda T(a\cdot T(m))\\
&=T(P(a)\cdot T(m))+T(a\cdot T(m))+\lambda T(a\cdot T(m)).
\end{array}$$
Applying $T$ again and using $T^2=T$, we obtain
$$(1+\lambda)T(a\cdot T(m))=0.
$$

\noindent
(\mref{it:3.3-2})
Since $P^2=P$, we have
$$\begin{array}{rllr}
P(a)\cdot T(m)&=P^2(a)\cdot T(m)\\
 &=T(P^2(a)\cdot m)+T(P(a)\cdot T(m))+\lambda T(P(a)\cdot m)\\
 &=T(P(a)\cdot m)+T(P(a)\cdot T(m))+\lambda T(P(a)\cdot m).
\end{array}$$
Applying $T$ to the equation and using $T^2=T$, we obtain
 $$
 T(P(a)\cdot m)+\lambda T(P(a)\cdot m)=0. $$
This gives Eq.~\eqref{eq:3.3-2}.
\smallskip

\noindent
(\mref{it:3.3-3})
Eq.~\eqref{eq:3.3-3} follows from the first two equations.
 \end{proof}

We also have the following doubling property of \rbped modules generalized from Rota-Baxter algebras.

\begin{prop} Let $(A, P)$ be a Rota-Baxter algebra of weight $\lambda$, and $(M,P,T)$ a \rbped $A$-module of weight $\lambda$. Define another binary operation $\star$ on $A$ by
$$a\star b=aP(b)+P(a)b+\lambda ab,
$$
and another operation $\triangleright$ between $A$ and $M$ by $$a\triangleright m=
P(a)\cdot m+a\cdot T(m)+\lambda (a\cdot m),
$$
for $a,b\in A,m\in M.$
Then the following conclusions hold.
\begin{enumerate}
\item
$T(a\triangleright m)=P(a)\cdot T(m)$ for any $a\in A,m\in M.$
\mlabel{it:3.6-1}
\item
$M$ with the operation $\triangleright$ is a nonunitary $(A,\star)$-module.
\mlabel{it:3.6-2}
\item
$(M,P,T,\triangleright)$ is a (nonunitary) \rbped $A$-module of weight $\lambda$.
\mlabel{it:3.6-3}
\end{enumerate}
\mlabel{thm:3.6}
\end{prop}

\begin{proof} (\mref{it:3.6-1}) It is direct to check
\begin{eqnarray}
\mlabel{eq:3.4}&P(a)\cdot T(m)=T(a\triangleright m).&
\end{eqnarray}

\noindent
(\mref{it:3.6-2}) We only need to prove that $(a\star b)\triangleright m=a\triangleright (b\triangleright m)$. Applying Eq.~\eqref{eq:3.4} and \cite[Theorem 1.1.17]{L. Guo2012}, we obtain
$$\begin{array}{rllr}
a\triangleright (b\triangleright m)&=a\triangleright (P(b)\cdot m+b\cdot T(m)+\lambda (b\cdot m))\\
 &=P(a)\cdot (P(b)\cdot m+b\cdot T(m)+\lambda (b\cdot m))+a\cdot \underbrace{T(P(b)\cdot m+b\cdot T(m)+\lambda (b\cdot m))}\\
&\ \ \ +\lambda (a\cdot (P(b)\cdot m+b\cdot T(m)+\lambda (b\cdot m)))\\
&=(P(a)P(b))\cdot m+(P(a)b)\cdot T(m)+\lambda(P(a)b)\cdot m+a\cdot T(b\triangleright m)+\lambda (aP(b))\cdot m\\
 &\ \ \ +(\lambda ab)\cdot T(m)+\lambda(\lambda ab)\cdot m\\
 &=(P(a)P(b))\cdot m+(P(a)b)\cdot T(m)+\lambda(P(a)b)\cdot m+(aP(b))\cdot T(m)+\lambda (aP(b))\cdot m\\
 &\ \ \ +(\lambda ab)\cdot T(m)+\lambda(\lambda ab)\cdot m\\
 &=(P(a)P(b))\cdot m+(\underbrace{P(a)b+aP(b)+\lambda ab})\cdot T(m)+\lambda((\underbrace{P(a)b+aP(b)+\lambda ab)\cdot m})\\
 &=P(a\star b)\cdot m+(a\star b)\cdot T(m)+\lambda((a\star b)\cdot m)\\
 &=(a\star b)\triangleright m.
\end{array}$$

\noindent
 (\mref{it:3.6-3}) By the definition of $\triangleright$ and Eq.(4) we have
 $$\begin{array}{rllr}
P(a)\triangleright T(m)&=P^2(a)\cdot T(m)+P(a)\cdot T^2(m)+\lambda(P(a)\cdot T(m))\\
 &=T(P(a)\triangleright m)+T(a\triangleright T(m))+\lambda T(a\triangleright m)
\end{array}$$
 as needed.
\end{proof}

\section{Constructions of \rbped modules}
\mlabel{sec:const}

In this section, we construct \rbped modules from Hopf algebras, weak Hopf algebras, Hopf modules, dimodules and weak Doi-Hopf modules, respectively.

We always work over a fixed field $\bfk$ and will use Sweedler's notation on coalgebras and comodules.
For a coalgebra $C$, we write its comultiplication
$\Delta(c)=c_{1}\otimes c_{2}$, for any $c\in C$; for a right (left)
$C$-comodule $M$, we denote its $C$-coaction by
$\rho(m)=m_{(0)}\otimes m_{(1)}$ ($\rho(m)=m_{(-1)}\otimes
m_{(0)}$), for any $m\in M$. Any unexplained definitions and
notations may be found in \cite{S. Montgomery1993} and \cite{M. Sweedler1969}.

\subsection{The construction on Hopf algebras}
\mlabel{ss:hopf}
In this subsection, we construct \grbped modules via Hopf algebras.

\begin{defn} A {\bf left action} of the bialgebra $H$ on the algebra $A$ is a
linear map
$$\alpha: H\otimes A\rightarrow A, \quad \alpha(h\otimes a)=h\cdot a\quad \tforall a,b\in A,\ h,g\in H,$$
such that

$(A1)\ h\cdot (ab)=(h_1\cdot a)(h_2\cdot b)$,

$(A2)\ 1_H\cdot a=a$,

$(A3)\ h\cdot(g\cdot a)=(hg)\cdot a$.

We will also call $A$ a {\bf left $H$-module algebra}.
\mlabel{de:1.4}
\end{defn}

As a direct consequence of Theorem~\mref{thm:3.2}, we demonstrate the following connection between the two meanings of an integral, one for a Hopf algebra and one for a Rota-Baxter algebra, both having its origin from integrals in analysis. See Example~\mref{ex:4.7} for the connection between cointegrals and \rbped modules.

Let $H$ be a bialgebra. If an element $x\in H$ satisfies $hx=\varepsilon(h)x$ for any $h\in H$, then we call $x$ a {\bf left integral} of $H$~\cite{M. Sweedler1969}.
%
Suppose that $H$ is a finite dimensional semisimple Hopf algebra. Then by \cite[Theorem 5.1.8]{M. Sweedler1969}, there exists a non-zero left integral $e$ such that $\varepsilon(e)=1$.

\begin{coro}
Let $H$ be a finite dimensional semisimple Hopf algebra. Let $A$ be a left $H$-module and define a map $T: A\rightarrow A$ by $T(a)=e\cdot a$. Then, $(A,T)$ is a \grbped $H$-module of  weight $-1$.
\mlabel{co:int}
\end{coro}

\begin{proof}
As a matter of fact, for any $a\in A, h\in H$, we have
$$\begin{array}{rllr}
T^2(a)&=T(e\cdot a)=e^2\cdot a=\varepsilon(e)e\cdot a=e\cdot a
=T(a),\\
T(h\cdot a)&=(eh)\cdot a=\varepsilon(h)e\cdot a
=h\cdot T(a).
\end{array}$$
Thus by Theorem~\mref{thm:3.2},  $(A,T)$ is a \grbped $H$-module of weight $-1$.
\end{proof}

Note that $\im\, T=A^H=\{a\in A|h\cdot a=\varepsilon(h)a,$ for any $h\in H\}$, the invariant subalgebra of $A$~\cite{M. Cohen1986}.

\begin{exam}\mlabel{ex:2.5}
\begin{enumerate}
\item
Let $A$ be a left $H$-module algebra. The {\bf smash product} $A\#H$ of $A$ and $H$ is defined to be the module $A\otimes H$ with multiplication given by
$$
(a\#h)(b\#g)=a(h_1\cdot b)\#h_2g \quad \tforall a, b\in A, h, g\in H.
$$
Then $A\#H$ is an associative algebra with unit $1_A\#1_H$~\cite{M. Sweedler1969}. It is obvious that $A$ and $H$ are subalgebras of the smash product $A\#H$.

Assume that $A$ is a left $H$-module algebra. Then we have the smash product $A\# H$, which is a left $H$-module via its multiplication. Hence that $(A\# H,\overline{T})$ is a \grbped $H$-module of weight $-1$, where $\overline{T}: A\# H\rightarrow A\# H$ is given by
$$
\overline{T}(a\# h)=e_1\cdot a\# e_2h.
$$
\item
In particular, if $A=H$ and $H$ is a finite dimensional Hopf algebra, it is naturally considered as a left $H$-module algebra via the adjoint action $h\rightharpoonup x=h_1xS(h_2)$. So by the above discussion, $(H\# H,\overline{T})$ is a \grbped $H$-module of weight $-1$, where $\overline{T}: H\# H\rightarrow H\# H$ is given by
$$
\overline{T}(g\# h)=e_1gS(e_2)\# e_3h.
$$
\end{enumerate}
\end{exam}

\begin{prop} Let $H$ be a Hopf algebra. Equip $H^*$ with the convolution product. Define a left action of $H^*$ on $H$:
$$\rhd: H^*\ot H\to H, f\rhd h=h_1f(h_2)  \quad \tforall f\in H^\ast, h\in H.$$
Then $H$ is a left $H^\ast$-module. Let $\chi\in H^\ast$ and define a map $T$ on $H$ by
$$T(h)=\chi(h_1) h_2.
$$
Then, $(H,T)$ is a \grbped $H^\ast$-module of weight $-1$ if and only if $\chi^2=\chi$.
\mlabel{pp:4.1}
\end{prop}

\begin{proof}
In fact, for any $f\in H^\ast, h\in H$,
$$\begin{array}{rllr}
f\rhd T(h)&=f\rhd(\chi(h_1)h_2)=\chi(h_1)f\rhd h_2\\
&=\chi(h_1)h_2  f(h_3) =T(h_1)  f(h_2) \\
&=T(f\rhd h).
\end{array}$$
So $T$ is a left $H^\ast$-module homomorphism. Thus by Theorem~\mref{thm:3.2}, $(H,T)$ is a \grbped $H^\ast$-module of weight $-1$ if and only if $T^2=T$.

It remains to prove that $T^2=T$ if and only if $\chi^2=\chi$. Assume that $\chi^2=\chi$. Then, we have
$$\begin{array}{rllr}
T^2(h)&=T(\chi(h_1)h_2)=\chi(h_1)\chi(h_2)h_3\\
&=\chi(h_1)h_2=T(h)
\end{array}$$
for any $h\in H$. Conversely, if $T^2=T$, then by the above expression,
$$\chi(h_1)\chi(h_2)h_3=\chi(h_1)h_2 \ \ \tforall h\in H.$$
This mean $\chi^2(h_1)h_2= \chi(h_1)h_2$ under the convolution product, that is, $\chi^2=\chi.$

Now the proof is completed.
 \end{proof}

\subsection{The construction via weak Hopf algebras}
\mlabel{ss:whopf}

In this subsection, we construct \rbped modules on module structures from weak bialgebras and quantum weak Hopf algebras. We first recall more backgrounds.

\begin{defn}\mlabel{de:1.6} Let $H$ be both an algebra and a coalgebra. Then $H$ is called a
{\bf weak bialgebra}~\cite{G. Bohm1999} if it satisfies the following conditions:
\begin{enumerate}
\item
$\Delta(xy)=\Delta(x)\Delta(y),$
\mlabel{it:1.6-1.1}
\item
$\varepsilon(xyz)=\varepsilon(xy_{1})\varepsilon(y_{2}z)
=\varepsilon(xy_{2})\varepsilon(y_{1}z),$
\mlabel{it:1.6-1.2}
\item
$\Delta^{2}(1_{H})=(\Delta(1_{H})\otimes1_{H})(1_{H}\otimes\Delta(1_{H}))=(1_{H}\otimes\Delta(1_{H}))(\Delta(1_{H})\otimes1_{H}),$
\\
for any $x,y,z\in H$, where $\Delta(1_{H})=1_{1}\otimes1_{2}$ and
$\Delta^{2}=(\Delta\otimes \id_{H})\circ\Delta$.
\mlabel{it:1.6-1.2}
\end{enumerate}
\end{defn}

For a weak bialgebra $H$, there are two linear maps
$\Pi^L,\Pi^R: H\rightarrow H$, called the {\bf target map} and {\bf source map}, respectively. They are defined by
$$\begin{array}{rllr}
\Pi^L(h)=\varepsilon(1_{1}h)1_{2},\quad
\Pi^R(h)=\varepsilon(h1_{2})1_{1} \ \ \tforall h\in H.
\end{array}$$
Their images will be denoted by $H^L$ and $H^R$, respectively.

By \cite{G. Bohm1999} the following equations hold for any $x,y\in H$.

$(W1)\ \ \Pi^L\circ\Pi^L=\Pi^L, \ \  \Pi^R\circ\Pi^R=\Pi^R,$

$(W2)\ \ \Pi^L(\Pi^L(x)y)=\Pi^L(x)\Pi^L(y), \ \  \Pi^R(x\Pi^R(y))=\Pi^R(x)\Pi^R(y),$

$(W3)\ \ \Pi^L(x)_1\otimes \Pi^L(x)_2=1_1\Pi^L(x)\otimes 1_2;\ \  \Pi^R(x)_1\otimes \Pi^R(x)_2=1_1\otimes \Pi^R(x)1_2,$

$(W4)\ \ \varepsilon(\Pi^R(x)y)=\varepsilon(xy)=\varepsilon(x\Pi^L(y)).$

\begin{prop}\mlabel{pp:4.3} Let $H$ be a weak bialgebra. Then $(H,\Pi^L)$ is a \grbped $H^L$-module of weight $-1$. Further $(H^L,\Pi^L)$ is a Rota-Baxter algebra of weight $-1$.
\end{prop}

\begin{proof}
First $H$ is a left $H^L$-module via the multiplication of $H$. By (W1) and (W2),   $\Pi^L: H\rightarrow H$ is idempotent and is a left $H^L$-module map. Thus by Theorem~\mref{thm:3.2}, $(H,\Pi^L)$ is a \grbped $H^L$-module of weight $-1$. Hence $(H^L,\Pi^L)$ is a \grbped $H^L$-module of weight $-1$.
In particular, $(H^L,\Pi^L,\Pi^L)$ is a \rbped $H^L$-module of weight $-1$. This means that $(H^L,\Pi^L)$ is a Rota-Baxter algebra of weight $-1$.
\end{proof}

Let $H$ be a weak bialgebra. If a linear map $S: H\rightarrow H$ satisfies the conditions:
$$
x_1S(x_2)=\Pi^L(x),\ \ S(x_1)x_2=\Pi^R(x),\ \ S(x_1)x_2S(x_3)=S(x)\ \
\tforall x\in H,
$$
then $S$ is called an {\bf antipode} of $H$ and $H$ is called a {\bf weak Hopf algebra}.
If so, then by \cite{Z.W. Wang2017}, the following equalities hold:

$(W5)\ \ \Pi^L\circ\Pi^R=\Pi^L\circ S=S\circ\Pi^R,\ \ \Pi^R\circ\Pi^L=\Pi^R\circ S=S\circ\Pi^L,$

$(W6)\ \ x_1\otimes \Pi^R(x_2)=x1_1\otimes S(1_2);\ \  \Pi^L(x_1)\otimes x_2=S(1_1)\otimes 1_2x,$ for any $x\in H.$

A weak Hopf algebra $H$ is called {\bf quantum commutative} if $h_1g\Pi^R(h_2)=hg$ for any $h,g\in H$. By~\cite[Proposition 4.1]{D. Bagio2015}, $H$ is quantum commutative if and only if $H^R\subseteq C(H)$.

\begin{prop} Let $H$ be a quantum commutative weak Hopf algebra with antipode $S$. Define an adjoint action of $H$ on $H$ by
\begin{center}
$h\rightharpoonup x=h_1xS(h_2)$ \ \ \ for all $h,x\in H$.
\end{center}
Then $(H,\Pi^L,\Pi^L)$ is a \rbped $H$-module of weight $-1$.
\mlabel{pp:4.4}
\end{prop}
Note that $(H,\Pi^L)$ is not a Rota-Baxter algebra.

\begin{proof} We first verify that $H$ is a left $H$-module under the adjoint action. In fact, for $h,g,x\in H$,
$$\begin{array}{rllr}
1\rightharpoonup x&=1_1xS(1_2)\stackrel{(W6)}=1_1x\Pi^R(1_2)=x,\\
h\rightharpoonup(g\rightharpoonup x)&=h_1g_1xS(g_2)S(h_2)\\
&=h_1g_1xS(h_2g_2)=(hg)\rightharpoonup x.
\end{array}$$

We next verify
\begin{equation}
\Pi^L(h\rightharpoonup x)=\Pi^L(h\rightharpoonup \Pi^L(x)) \ \ \tforall h,x\in H.
\mlabel{eq:pi1}
\end{equation}

This is so since for the left hand side, we have
$$\begin{array}{rllr}
\Pi^L(h\rightharpoonup x)&=\varepsilon(1_1h_1xS(h_2))1_2\\
&\stackrel{(W4)}=\varepsilon(\Pi^R(1_1h_1)x\Pi^L(S(h_2)))1_2\\
&\stackrel{(W5)}=\varepsilon(\Pi^R(1_1h_1)x\Pi^L(\Pi^R(h_2)))1_2\\
&\stackrel{(W6)}=\varepsilon(\Pi^R(1_1h1^\prime_1)x\Pi^L(S(1^\prime_2)))1_2\\
&=\varepsilon(1_1h1^\prime_1xS(1^\prime_2))1_2\\
&=\varepsilon(1_11^\prime_1hxS(1^\prime_2))1_2\\
&=\varepsilon(1_1hx)1_2\\
&=\Pi^L(hx),
\end{array}$$
where $\Delta(1)=1^\prime_1\otimes 1^\prime_2.$

For the right hand side, by~\cite[Corollary 2.2]{L.Y. Zhang2010}, $h\rightharpoonup \Pi^L(x)=h_1\Pi^L(x)S(h_2)=\Pi^L(hx)$ for $h,g,x\in H$. So by $(W1)$, we obtain $$\Pi^L(h\rightharpoonup \Pi^L(x))=\Pi^L(hx). $$
This proves Eq.~\eqref{eq:pi1}.

Further we have, for any $h,x\in H$,
$$\begin{array}{rllr}
\Pi^L(h)\rightharpoonup\Pi^L(x)&=\Pi^L(h)_1\Pi^L(x)S(\Pi^L(h)_2)\\
&\stackrel{(W3)}=1_1\Pi^L(h)\Pi^L(x)S(1_2)\\
&=\Pi^L(h)\Pi^L(x),\\
\Pi^L(\Pi^L(h)\rightharpoonup x)&=\Pi^L(\Pi^L(h)x)\stackrel{(W2)}=\Pi^L(h)\Pi^L(x),
\end{array}$$
yielding $\Pi^L(h)\rightharpoonup\Pi^L(x)=\Pi^L(\Pi^L(h)\rightharpoonup x)$.

Therefore we obtain
$$
\Pi^L(h)\rightharpoonup\Pi^L(x)=\Pi^L(\Pi^L(h)\rightharpoonup x)+\Pi^L(h\rightharpoonup\Pi^L(x))-\Pi^L(h\rightharpoonup x),
$$
that is, $(H,\Pi^L,\Pi^L)$ is a \rbped $H$-module of weight $-1$.
\end{proof}

\subsection{The construction on Hopf modules}
\mlabel{ss:hopfmod}

In this subsection, we construct \rbped modules via an antipode $S$ of the Hopf algebra $H$.

\begin{prop} Let $H$ be a Hopf algebra with an antipode $S$, and $M$ a right $H$-Hopf-module. Define a map $E_M: M\rightarrow M$ given by
$$
E_M(m)=m_{(0)}\cdot S(m_{(1)}) \quad \tforall m\in M.
$$
Then, the following statements hold.
\begin{enumerate}
\item Define
$$\widetilde{\varepsilon}_H: H\rightarrow H, \ \ \widetilde{\varepsilon}_H(h)=\varepsilon_H(h)1_H.
$$
Then $(M,\widetilde{\varepsilon}_H,E_M)$ is a right \rbped $H$-module of weight $-1$.
\mlabel{it:4.5-1}
\item
Define a left action of $H^*$ on $M$ by $f\rightharpoonup m=m_{(0)}f(m_{(1)})$ for $f\in H^\ast, m\in M$. Then $(M,E_M)$ is a \grbped $H^\ast$-module of weight $-1$.
\mlabel{it:4.5-2}
\end{enumerate}
\mlabel{pp:4.5}
\end{prop}

\begin{proof} Since $M$ is a right $H$-Hopf-module, we derive $E_M^2=E_M$ and
\begin{center}
$\im\, E_M\subseteq M^{coH}=\{m\in M\,|\,\rho(m)=m\otimes 1\}$.
\end{center}

\noindent
(\mref{it:4.5-1}) For any $h\in H$ and $m\in M$, we have
$$\begin{array}{rllr}
E_M(m\cdot h)&=(m_{(0)}\cdot h_1)\cdot S(m_{(1)}h_2)\\
&=m_{(0)}\cdot(h_1S(h_2)S(m_{(1)}))\\
&=m_{(0)}\cdot S(m_{(1)})\varepsilon_H(h)\\
&=E_M(m)\varepsilon_H(h).
\end{array}$$
Thus
$$\begin{array}{rllr}
E_M(E_M(m)\cdot h)&=E^2_M(m)\varepsilon_H(h)=E_M(m)\varepsilon_H(h)\\
&=E_M(m)\cdot\widetilde{\varepsilon}(h),\\
E_M(m\cdot\widetilde{\varepsilon}_H(h))&=E_M(m)\varepsilon_H(h)=E_M(m\cdot h).
\end{array}$$
Hence we obtain
$$
E_M(m)\cdot\widetilde{\varepsilon}_H(h)=E_M(E_M(m)\cdot h)+E_M(m\cdot\widetilde{\varepsilon}_H(h))-E_M(m\cdot h)
$$
for any $h\in H, m\in M$.
\smallskip

(\mref{it:4.5-2}) Since $M$ is a right $H$-comodule, $(M,\rightharpoonup)$ is a left $H^\ast$-module. So for any $f\in H^\ast$ and $m\in M,$
$$\begin{array}{rllr}
f\rightharpoonup E_M(m)&=f\rightharpoonup(m_{(0)}\cdot S(m_{(1)}))\\
&=  f\big((m_{(0)}\cdot S(m_{(1)}))_{(1)}\big) (m_{(0)}\cdot S(m_{(1)}))_{(0)}\\
&=  f\big(m_{(0)(1)}S(m_{(1)1})\big) m_{(0)(0)}\cdot S(m_{(1)2})\\
&=  f(m_{(1)1}S(m_{(1)2})) m_{(0)}\cdot S(m_{(1)3})\\
&=  f(1_H) m_{(0)}\cdot S(m_{(1)})\\
&=  f(1_H) E_M(m)\\
&=  f(E_M(m)_{(1)}) E_M(E_M(m)_{(0)})\\
&=E_M(f\rightharpoonup E_M(m)).
\end{array}$$
Thus $E_M$ is a left $H^\ast$-module map.
Again since $E_M^2=E_M$, by Theorem~\mref{thm:3.2}, $(M,E_M)$ is a \grbped $H^\ast$-module of weight $-1$.
\end{proof}

\subsection{The construction on dimodules}
\mlabel{ss:dimod}

In this subsection, we construct \rbped modules via a functional in the dual algebra $H^\ast$ of the bialgebra $H$.

\begin{defn} Assume that $H$ is a bialgebra. A $\bfk$-module $M$ which is both a left $H$-module and a right $H$-comodule is called a {\bf left, right $H$-dimodule}~\cite{S. Caenepeel1994} if for any $h\in H, m\in M$, the following equality holds:

\begin{equation}
 \rho(h\cdot m)=h\cdot m_{(0)}\otimes m_{(1)},
 \mlabel{eq:D}
 \end{equation}
where $\rho$ is the right $H$-comodule structure map of $M$.
\mlabel{de:1.5}
\end{defn}

\begin{prop}
Let $H$ be a bialgebra and let $f\in H^\ast$. Let $M$ be a left, right $H$-dimodule. Define a map $T: M\rightarrow M$ by taking
$T(m)=m_{(0)}f(m_{(1)}).$
Then the following statements are equivalent.
\begin{enumerate}
\item
The pair $(M,T)$ is a \grbped $H$-module of weight $-1$;
\mlabel{it:4.6-1}
\item
There is a linear operator $P:H\to H$ such that  $(M,P,T)$ is a \rbped $H$-module of weight $-1$;
\item
$T$ is idempotent;
\mlabel{it:4.6-2}
\item
$f$ is idempotent under the convolution product.
\end{enumerate}
\mlabel{pp:4.6}
\end{prop}

\begin{proof} Since $M$ is a left, right $H$-dimodule, Eq.~(\mref{eq:D}) gives
$$
T(h\cdot m)=(h\cdot m)_{(0)}f((h\cdot m)_{(1)})
=h\cdot(m_{(0)}f(m_{(1)}))
=h\cdot T(m)\ \ \tforall h\in H, m\in M.
$$
 Hence $T$ is a left $H$-module map. Thus by Theorem~\mref{thm:3.2}, the first three statements are equivalent. Furthermore,

$$T(m)=m_{(0)}f(m_{(1)}).$$
$$ T^2(m)=T(m_{(0)}f(m_{(1)}))= T(m_{(0)})f(m_{(1)})
=m_{(0)(0)}f(m_{(0)(1)})f(m_{(1)})=m_{(0)}f^2(m_{(1)}).$$
Thus $T^2=T$ if and only if $f^2=f$.
\end{proof}

Let $H$ be a bialgebra. If there exists an element $\lambda\in H^\ast$ such that $f\lambda=\varepsilon_{H^\ast}(f)\lambda$ for any $f\in H^\ast$, then, we call $\lambda$ a {\bf left cointegral} of $H^\ast$. Furthermore, if $\varepsilon_{H^\ast}(\lambda)=1$, then, we call $H$ a {\bf cosemisimple bialgebra}.
In this case, there exists a cointegral $\chi\in H^\ast$ such that $\chi(1_H)=1$. Since $\chi^2=\chi$, by Proposition~\mref{pp:4.6} we obtain

\begin{coro}
Let $M$ be a left, right $H$-dimodule and let $T: M\rightarrow M$ be defined by
$$T(m)=m_{(0)}\chi(m_{(1)}).
$$
Then
$(M,T)$ is a \grbped $H$-module of weight $-1$,
\mlabel{it:4.7-1}
\end{coro}

In the following examples, we construct \rbped modules using cointegrals in $H^\ast$.

\begin{exam}\mlabel{ex:4.7}
Suppose that $H$ is a cosemisimple bialgebra with cointegral $\chi$ satisfying $\chi(1_H)=1$.
\begin{enumerate}
\item
A {\bf skew pairing Long bialgebra}~\cite{Mi,L.Y. Zhang2006}
is a pair $(H,\sigma)$ consisting of a bialgebra  $H$ and a linear map $\sigma: H\otimes H\rightarrow k$,  satisfying the following conditions: for any $x,y,z\in H,$

$(L1)\ \ \sigma(x_1,y)x_2=\sigma(x_2,y)x_1,$

$(L2)\ \ \sigma(x,1)=\varepsilon(x),$

$(L3)\ \ \sigma(x,yz)=\sigma(x_2,y)\sigma(x_1,z),$

$(L4)\ \ \sigma(1,x)=\varepsilon(x),$

$(L5)\ \ \sigma(xy,z)=\sigma(x,z_1)\sigma(y,z_2).$

Assume that $(H,\sigma)$ is a skew pairing Long bialgebra. Define $x\rightharpoonup h=\sigma(h_2,x)h_1$ for $h,x\in H$. Then, by \cite[Theorem 3.3]{L.Y. Zhang2006}, we know that $(H,\rightharpoonup,\Delta)$ is a left, right $H$-dimodule.
Thus, defining $T: H\rightarrow H$ by
$$T(h)=h_1\chi(h_2),$$
then by Corollary~\mref{it:4.7-1}, $(H,T)$ is a \grbped $H$-module of weight $-1$.
\item
A {\bf braided bialgebra}~\cite{S. Montgomery1993} is a pair $(H,\sigma)$ consisting of a bialgebra $H$ and a linear map $\sigma: H\otimes H\rightarrow k$, satisfying the following conditions: for any $x,y,z\in H,$

$(B1)\ \ \sigma(x_1,y_1)y_2x_2=x_1y_1\sigma(x_2,y_2),$

$(B2)\ \ \sigma(x,yz)=\sigma(x_1,y)\sigma(x_2,z),$

$(B3)\ \ \sigma(xy,z)=\sigma(x,z_2)\sigma(y,z_1).$

Let $(H,\sigma)$ be a braided bialgebra. Define $x\rightharpoonup h=\sigma(x,h_1)h_2$. Then by~\cite[Example 2.3]{L.Y. Zhang2003}, $(H,\rightharpoonup,\Delta)$ is a left, right $H$-dimodule.
So, by Corollary~\mref{it:4.7-1}, $(H,T)$ is a \grbped $H$-module of weight $-1$, where $T: H\rightarrow H$ is given by
$$T(h)=h_1\chi(h_2).
$$
\item
A bialgebra $H$ is called {\bf quasitriangular}~\cite{S. Montgomery1993} if there is an invertible element $\mathcal{R}^{-1}$ of $\mathcal{R}=\mathcal{R}_i\otimes \mathcal{R}_j\in H\otimes H$ such that for any $h\in H$,

$(Q1)\ \ \mathcal{R}\Delta(h)\mathcal{R}^{-1}=\tau\circ\Delta(h),$

$(Q2)\ \ (\Delta\otimes id)(\mathcal{R})=\mathcal{R}^{13}\mathcal{R}^{23},$

$(Q3)\ \ (id\otimes\Delta)(\mathcal{R})=\mathcal{R}^{13}\mathcal{R}^{12},$\\
where $\tau$ is a twisted map in $H\otimes H$, and $\mathcal{R}^{13}=\mathcal{R}_i\otimes 1\otimes \mathcal{R}_j, \mathcal{R}^{23}=1\otimes \mathcal{R}_i\otimes 1\otimes \mathcal{R}_j, \mathcal{R}^{12}=\mathcal{R}_i\otimes  \mathcal{R}_j\otimes 1.$

Let $(H,\mathcal{R})$ be a quasitriangular bialgebra. Define a coaction $\rho: H\rightarrow H\otimes H$ by $\rho(h)=h\mathcal{R}_i\otimes \mathcal{R}_j$. Then, $(H,\rho,m)$ is a left, right $H$-dimodule:
$$\begin{array}{rllr}
(\rho\otimes \id)\rho(h)&=(\rho\otimes \id )(h\mathcal{R}_i\otimes \mathcal{R}_j)\\
&=(h\mathcal{R}_i)\mathcal{R^\prime}_i\otimes \mathcal{R^\prime}_j\otimes \mathcal{R}_j)\ \ (\mathcal{R}=\mathcal{R^\prime}_i\otimes \mathcal{R^\prime}_j)\\
&=h\mathcal{R}_i\otimes \mathcal{R}_{j1}\otimes \mathcal{R}_{j2}\\
&=(id\otimes \Delta)\rho(h),\\
\rho(x\cdot h)&=xh\mathcal{R}_i\otimes \mathcal{R}_j=x\cdot(h\mathcal{R}_i)\otimes \mathcal{R}_j
=x\cdot h_{(0)}\otimes h_{(1)}.
\end{array}$$
So, by Corollary~\mref{it:4.7-1}, $(H,T)$ is a \grbped $H$-module of weight $-1$, where $T: H\rightarrow H$ is given by $$T(h)=h\mathcal{R}_i\chi(\mathcal{R}_j).
$$
\end{enumerate}
\end{exam}

\subsection{The construction on weak Doi-Hopf modules}
\mlabel{ss:doihopf}

In this subsection, we construct \rbped modules from weak Doi-Hopf modules on a weak Hopf algebra via the antipode $S$.

\begin{defn}\mlabel{de:1.7} Let $H$ be a weak bialgebra. A $\bfk$-algebra $A$ is called a weak right $H$-comodule algebra in~ \cite{Bo}, if there exists a weak right coaction $\rho$ of $H$ on $A$ which is also an algebra map. That is, the map $\rho :A\rightarrow A\otimes H$ such that for any $a,b \in A$:

$(WCA1)\ \ (\rho\otimes id_H)\rho=(id_A\otimes \Delta)\rho,$

$(WCA2)\ \rho(1_A)(a\otimes 1_H)=(id_A\otimes \Pi^L)\rho(a),$

$(WCA3)\ \rho(ab)=\rho(a)\rho(b).$
\end{defn}

\begin{defn}\mlabel{de:1.8} Let $H$ be a weak bialgebra and $A$ a weak right $H$-comodule algebra.  A $\bfk$-module $M$ which is both a right $A$-module and a right $H$-comodule is called a {\bf weak right $(A,H)$-Doi-Hopf-module}~ \cite{Bo,L.Y. Zhang2010b} if for any $h\in H, m\in M$,

$(WH)\ \ \rho_M(m\cdot a)=m_{(0)}\cdot a_{(0)}\otimes m_{(1)}a_{(1)},$
\\
where $\rho_M$ is the right $H$-comodule structure map of $M$.
\end{defn}

In the following, we will always assume that $H$ is a weak Hopf algebra, $A$ is a weak right $H$-comodule algebra, and $M$ in $\mathcal{M}_A^H$, the category of weak right $(A,H)$-Doi-Hopf modules.
It is obvious that $H$ is a weak right $H$-comodule algebra whose comodule structure map is given by its comultiplication $\Delta_H$.

\begin{theorem} Assume that $H$ is a weak Hopf algebra and $A$ a weak right $H$-comodule algebra. If $\phi: H\rightarrow A$ is a weak right $H$-comodule algebra map and $M\in \mathcal{M}_A^H$, define
$$E_A: A\rightarrow A, \quad E_A(a)=a_{(0)}\phi S(a_{(1)}) \quad \tforall a\in A
$$
 and
$$E_M: M\rightarrow M, \quad E_M(m)=m_{(0)}\cdot\phi S(m_{(1)}) \quad  \tforall m\in M.
 $$
Then, $(M,E_A,E_M)$ is a right \rbped $A$-module of weight $-1$.
\mlabel{thm:4.8}
\end{theorem}
Note that, unlike in the previous cases, Theorem~\mref{thm:3.2} does not apply in the proof of this theorem.

\begin{proof} Following an argument in~\cite{L.Y. Zhang2010}, we can show that
$$E_M(m)\in M^{coH}=\{x\in M\,|\,\rho(x)=x_{(0)}\otimes \Pi^L(x_{(1)})\}
\quad \tforall m\in M.
$$
In fact, for any $m\in M$,
$$\begin{array}{rllr}
\rho(E_M(m))&=\rho(m_{(0)}\cdot\phi S(m_{(1)}))\\
&=(m_{(0)}\cdot\phi S(m_{(1)}))_{(0)}\otimes(m_{(0)}\cdot\phi S(m_{(1)}))_{(1)}\\
&=m_{(0)(0)}\cdot\phi S(m_{(1)})_{(0)}\otimes m_{(0)(1)}\phi S(m_{(1)})_{(1)}\\
&=m_{(0)}\cdot\phi S(m_{(1)3})\otimes m_{(1)1}S(m_{(1)2})\\
&=m_{(0)}\cdot\phi S(m_{(1)2})\otimes \Pi^L(m_{(1)1}).
\end{array}$$
So by $(W1)$, we have
$E_M(m)_{(0)}\otimes E_M(m)_{(1)}=E_M(m)_{(0)}\otimes \Pi^L(E_M(m)_{(1)})$, that is,
$E_M(m)\in M^{coH}.$

Moreover, for any $m\in M, a\in A$, we have
$$\begin{array}{rllr}
E_M(m\cdot a)&=(m\cdot a)_{(0)}\cdot\phi S((m\cdot a)_{(1)})\\
&=(m_{(0)}\cdot a_{(0)})\cdot\phi S((m_{(1)}a_{(1)}))\\
&=(m_{(0)}\cdot a_{(0)})\cdot(\phi S(a_{(1)})\phi S(m_{(1)}))\\
&=m_{(0)}\cdot (a_{(0)}\phi S(a_{(1)})\phi S(m_{(1)}))\\
&=m_{(0)}\cdot (E_A(a)\phi S(m_{(1)})).
\end{array}$$
So by $E_A^2=E_A$, we obtain
$$
E_M(m\cdot E_A(a))=E_M(m\cdot a).
$$

Again by $(W5)$ and the equality $E_A(a)\phi(\Pi^R(h))=\phi(\Pi^R(h))E_A(a)$ for any $a\in A, h\in H$ (see Lemma 3.1 in \cite{L.Y. Zhang2010}), we have
$$\begin{array}{rllr}
E_M(E_M(m)\cdot a)&=E_M(m)_{(0)}\cdot (E_A(a)\phi S(E_M(m)_{(1)}))\\
&=m_{(0)}\cdot\phi S(m_{(1)2})\cdot (E_A(a)\phi S(\Pi^L(m_{(1)})))\\
&=m_{(0)}\cdot(\phi S(m_{(1)2})\phi S(\Pi^L(m_{(1)}))E_A(a))\\
&=m_{(0)}\cdot(\phi S(\Pi^L(m_{(1)})m_{(1)2})E_A(a))\\
&=m_{(0)}\cdot(\phi S(m_{(1)})E_A(a))\\
&=E_M(m)\cdot E_A(a).
\end{array}$$

Hence
$$
E_M(m)\cdot E_A(a)=E_M(E_M(m)\cdot a)+E_M(m\cdot E_A(a))-E_M(m\cdot a),
$$
for any $m\in M, a\in A$, that is, $(M,E_A,E_M)$ is a right \rbped $A$-module of weight $-1$.
\end{proof}

Applying the above theorem, we have

\begin{coro}\mlabel{co:4.9} Assume that $H$ is a weak Hopf algebra, $A$ a weak right $H$-comodule algebra and $\phi: H\rightarrow A$ a right $H$-comodule algebra map. Let $E_A$ and $E_M$ be as defined in Theorem~\mref{thm:4.8}.
Then, for any right $(A,H)$-Doi-Hopf module $M\in \mathcal{M}_A^H$, the triple $(M,E_A,E_M)$ is a right \rbped $A$-module of weight $-1$.
\end{coro}

\begin{remark}\mlabel{rk:4.10}
\begin{enumerate}
\item
In the above theorem, by taking $M=A$, we have $A\in \mathcal{M}_A^H$ whose module structure is given by the multiplication of the weak comodule algebra $A$. So from a weak right $H$-comodule algebra map $\phi: H\rightarrow A$, we obtain a right \rbped $A$-module $(A,E_A,E_A)$ of weight $-1$, where $E_A: A\rightarrow A$ is given by
$$E_A(a)=a_{(0)}\phi S(a_{(1)}).$$

Therefore, $(A,E_A)$ is a Rota-Baxter algebra of weight $-1$.

\item
In the above theorem, by taking $A=H$, it is easy to see that $\id: H\rightarrow H$ is a right $H$-comodule algebra map. So for any weak $H$-Hopf module $M$, $(M,E_H,E_M)$ is a right \rbped $A$-module of weight $-1$.

Note here that $E_H: H\rightarrow H$ is given by
 $$E_H(h)=\Pi^L(h). $$

\item
It is obvious that $H\in \mathcal{M}_H^H$ has its module and comodule structures are given by its multiplication and comultiplication. So by above, for any weak Hopf algebra $H$, $(H,E_H,E_H)$ is a right \rbped $H$-module of weight $-1$.
Hence  $(H,E_H)$ is a Rota-Baxter algebra of weight $-1$.
\end{enumerate}
\end{remark}

\noindent {\bf Acknowledgements}: This work was supported by the National Natural Science Foundation of China (Grant Nos.~11371178 and ~11571173).


\end{document}